\documentclass[12pt]{amsart}
\usepackage{amsmath,amsthm,amsfonts,amssymb,mathrsfs}
\usepackage{todonotes}

\date{\today}

\usepackage{color}
\input xymatrix
\xyoption{all}

\usepackage{hyperref}

\newtheorem{theorem}{Theorem}[section]

\newtheorem{question}{Question}
\newtheorem{proposition}[theorem]{Proposition}
\newtheorem{corollary}[theorem]{Corollary}
\newtheorem{lemma}[theorem]{Lemma}
\theoremstyle{definition}
\newtheorem{example}[theorem]{Example}
\newtheorem{problem}[theorem]{Problem}
\newtheorem{remark}[theorem]{Remark}

\newcommand\w{\omega}
\newcommand{\N}{\mathbb N}

\def\N{\mathbb N}
\def\t_c{\tau_{comp}}

\begin{document}

\title[Closed subsets of compact-like topological spaces]{Closed subsets of compact-like topological spaces}

\author[S.~Bardyla]{Serhii~Bardyla}
\address{S.~Bardyla: Institute of Mathematics, Kurt G\"{o}del Research Center, Vienna, Austria}
 \email{sbardyla@yahoo.com}
 \thanks{The work of the first author is supported by the Austrian Science Fund FWF (Grant  I
3709 N35).}

\author[A.~Ravsky]{Alex~Ravsky}
\address{A.~Ravsky: Pidstryhach Institute for Applied Problems of Mechanics and Mathematics,
Nat. Acad. Sciences of Ukraine, Lviv, Ukraine}
\email{alexander.ravsky@uni-wuerzburg.de}



\begin{abstract}
We investigate closed subsets (subsemigroups, resp.) of compact-like topological spaces
(semigroups, resp.). We prove that each Hausdorff topological space can be embedded as a closed
subspace into an H-closed topological space. However, the semigroup of $\w{\times\w}$-matrix units
cannot be embedded into a topological semigroup which is a weakly H-closed topological space. We show
that each Hausdorff topological space is a closed subspace of some $\omega$-bounded pracompact
topological space and describe open dense subspaces of countably pracompact topological spaces.
Also, we construct a pseudocompact topological semigroup which contains the bicyclic monoid as a
closed subsemigroup, providing a positive solution of a problem posed by Banakh, Dimitrova, and
Gutik.
\end{abstract}
\maketitle

\section{Preliminaries}
In this paper all topological spaces are assumed to be Hausdorff.
By $\omega$ we denote the first infinite cardinal. For ordinals $\alpha,\beta$ put $\alpha\leq
\beta$, ($\alpha< \beta$, resp.) iff $\alpha \subset \beta$ ($\alpha \subset \beta$ and
$\alpha\ne\beta$ , resp.). By $[\alpha,\beta]$ ($[\alpha,\beta)$, $(\alpha,\beta]$, $(\alpha,\beta)$, resp.) we denote the set of all
ordinals $\gamma$ such that $\alpha\leq \gamma\leq\beta$ ($\alpha\leq \gamma<\beta$, $\alpha< \gamma\leq\beta$, $\alpha<\gamma<\beta$, resp.). The
cardinality of a set $X$ is denoted by $|X|$.

For a subset $A$ of a topological space $X$ by
$\overline{A}$ we denote the closure of the set $A$ in $X$.

 A family $\mathcal{F}$ of subsets of a set $X$ is called a {\em filter} if it satisfies the following conditions:
\begin{itemize}
\item[$(1)$] $\emptyset\notin \mathcal{F}$;
\item[$(2)$] If $A\in \mathcal{F}$ and $A\subset B$ then $B\in \mathcal{F}$;
\item[$(3)$] If $A,B\in \mathcal{F}$ then $A\cap B\in\mathcal{F}$.
\end{itemize}
A family $\mathcal{B}$ is called a {\em base} of a filter $\mathcal{F}$ if for each element
$A\in\mathcal{F}$ there exists an element $B\in\mathcal{B}$ such that $B\subset A$. A filter on a
topological space $X$ is called an {\em $\w$-filter} if it has a countable base.  A filter
$\mathcal{F}$ is called {\em free} if $\bigcap\mathcal{F}=\varnothing$.
A filter
on a topological space $X$ is called {\em open} if it has a base which consists of open subsets. A
point $x$ is called an {\em accumulation point} ({\em $\theta$-accumulation point}, resp.) of a
filter $\mathcal{F}$ if for each open neighborhood $U$ of $x$ and for each $F\in\mathcal{F}$ the
set $U\cap F$ ($\overline{U}\cap F$, resp.) is non-empty. A topological space $X$ is said to be

\begin{itemize}
  \item \emph{compact}, if each filter has an accumulation point;
  \item \emph{sequentially compact}, if each sequence $\{x_n\}_{n\in\w}$ of points of $X$
  has a convergent subsequence;
  \item \emph{$\omega$-bounded}, if each countable subset of $X$ has compact closure;
  \item \emph{totally countably compact}, if each sequence of $X$ contains a subsequence with
  compact closure;
   \item \emph{countably compact}, if each infinite subset $A\subseteq X$  has  an  accumulation  point;
   \item \emph{$\omega$-bounded pracompact}, if there exists a dense subset $D$ of $X$  such that each countable subset of the set $D$ has compact closure in $X$;
  \item \emph{totally countably pracompact}, if there exists a dense subset $D$ of $X$  such that each sequence of points of the set $D$ has a subsequence with compact closure in $X$;
   \item \emph{countably pracompact}, if there exists a dense subset $D$ of $X$  such that every infinite subset $A\subseteq D$  has  an  accumulation  point in $X$;
   \item \emph{pseudocompact}, if $X$ is Tychonoff and each real-valued function on $X$ is bounded;
   \item \emph{H-closed}, if each filter on $X$ has a $\theta$-accumulation point;
  \item \emph{feebly $\omega$-bounded}, if for each sequence $\{U_n\}_{n\in\w}$  of non-empty  open  subsets of $X$ there is a compact subset $K$ of
  $X$ such that $K\cap U_n\ne\varnothing$ for each $n\in \w$;
\item \emph{totally feebly compact}, if for each sequence $\{U_n\}_{n\in\w}$
of non-empty  open  subsets of $X$ there is a compact subset $K$ of
$X$ such that $K\cap U_n\ne\varnothing$ for infinitely many $n\in \w$;
  \item \emph{selectively feebly compact},
if for each sequence $\{U_n\}_{n\in\w}$ of non-empty open subsets of $X$, for each $n\in\w$
we can choose a point $x_n\in U_n$
such that the sequence $\{x_n:n\in \w\}$ has an accumulation point.
  \item \emph{feebly compact}, if each open $\omega$-filter on $X$ has an accumulation point.
 \end{itemize}


The interplay between some of the above properties is shown in the diagram at page 3 in~\cite{Gutik-Ravsky}.
\begin{remark}
H-closed topological spaces has few different equivalent definitions. For a topological space $X$ the following conditions are equivalent:
\begin{itemize}
\item $X$ is H-closed;
\item if $X$ is a subspace of a Hausdorff topological space $Y$, then $X$ is closed in $Y$;
\item each open filter on $X$ has an accumulation point;
\item for each open cover $\mathcal{F}=\{F_{\alpha}\}_{\alpha\in A}$ of $X$ there exists a finite
subset $B\subset A$ such that $\cup_{\alpha\in B}\overline{F_\alpha}=X$.
  \end{itemize}
H-closed topological spaces in terms of $\theta$-accumulation points were investigated in~\cite{Dik,Jos1,Jos2,Os,Por1,Por2,Vel,Ver}.
Also recall that each H-closed space is feebly compact.
\end{remark}

In this paper we investigate closed subsets (subsemigroups, resp.) of compact-like topological spaces
(semigroups, resp.). We prove that each Hausdorff topological space can be embedded as a closed
subspace into an H-closed topological space. However, the semigroup of
$\w{\times\w}$-matrix units cannot be embedded into a topological semigroup which is a weakly H-closed
topological space. We show that each Hausdorff topological space is a closed subspace of some $\omega$-bounded pracompact topological space and describe open dense subspaces of countably pracompact topological spaces.  Also, we construct a pseudocompact topological semigroup which contains the
bicyclic monoid as a closed subsemigroup, providing a positive solution of Problem~\ref{problem}.

\section{Closed subspaces of compact-like topological spaces}

The productivity of compact-like properties is a known topic in general topology.
According to Tychonoff's theorem, a (Tychonoff) product of a family of compact spaces is compact,
On the other hand, there are two countably compact
spaces whose product is not feebly compact (see~\cite{Engelking-1989},
the paragraph before Theorem 3.10.16). The product of a countable family of sequentially
compact spaces is sequentially compact~\cite[Theorem~3.10.35]{Engelking-1989}.
But already the Cantor cube $D^\mathfrak{c}$ is not sequentially compact
(see~\cite{Engelking-1989}, the paragraph after Example 3.10.38).
On the other hand some compact-like properties are also preserved by products, see
~\cite[$\S$ 3-4]{VaughanHSTT} (especially Theorem 3.3, Proposition 3,4, Example 3.15, Theorem 4.7, and Example 4.15)
and $\S$7 for the history, ~\cite[$\S$ 5]{StephensonJr1984},
and~\cite[Sec. 2.3]{Gutik-Ravsky}.

\begin{proposition}\label{prop:fwb-prod} A product of any family of feebly $\omega$-bounded spaces is feebly $\omega$-bounded.
\end{proposition}
\begin{proof}
Let $X=\prod\{X_\alpha\colon \alpha\in A\}$ be a product of a family of feebly $\omega$-bounded
spaces and let $\{U_n\}_{n\in\omega}$ be a family of non-empty open subsets of the space $X$. For
each $n\in\omega$ let $V_n$ be a basic open set in $X$ which is contained in $U_n$. For each
$n\in\omega$ and $\alpha\in A$ let $V_{n,\alpha}=\pi_{\alpha}(V_n)$ where by $\pi_{\alpha}$ we
denote the projection on $X_{\alpha}$. For each $\alpha\in A$ there exists a compact subset
$K_\alpha$ of $X_\alpha$, intersecting each $V_{n,\alpha}$. Then the set $K=\prod\{K_\alpha\colon
\alpha\in A\}$ is a compact subset of $X$ intersecting each $V_n\subset U_n$.
\end{proof}


A non-productive compact-like properties still can be preserved by products with
more strong compact-like spaces. For instance, a product of a countably compact space
and a countably compact $k$-space or a sequentially compact space is countably compact,
and a product of a pseudocompact space and a pseudocompact $k$-space
or a sequentially compact Tychonoff space is pseudocompact (see \cite[Sec. 3.10]{Engelking-1989}).

\begin{proposition}\label{prop:cp-prod} A product $X\times Y$ of a countably pracompact space $X$ and a
totally countably pracompact space $Y$ is countably pracompact.
\end{proposition}
\begin{proof}
Let $D$ be a dense subset of $X$ such that each infinite subset of $D$ has an accumulation
point in $X$ and $F$ be a dense subset of $Y$ such that
each sequence of points of the set $F$ has a subsequence
contained in a compact set.
Then $D{\times} F$ is a dense subset of $X{\times} Y$. So to prove that the space
$X{\times} Y$ is countably pracompact it suffices to show that each sequence
$\{(x_n,y_n)\}_{n\in\omega}$ of points of $D{\times} F$ has an accumulation point.
Taking a subsequence, if needed, we can assume that a
sequence $\{y_n\}_{n\in\omega}$ is contained in a compact set $K$.
Let $x\in X$ be an accumulation point of a sequence $\{x_n\}_{n\in \omega}$
and $\mathcal B(x)$ be the family of neighborhoods of the point $x$.
For each $U\in \mathcal B(x)$ put $Y_U=\overline{\{y_n\mid x_n\in U\}}$.
Then $\{Y_U \mid U\in \mathcal B(x)\}$ is a centered family of closed subsets of
a compact space $K$, so there exists a point $y\in\bigcap \{Y_U\mid U\in \mathcal B(x)\}$.
Clearly, $(x,y)$ is an accumulation point of the sequence
$\{(x_n,y_n)\}_{n\in\omega}$.
\end{proof}

\begin{proposition}\label{prop:sfc-prod} A product $X\times Y$ of a selectively feebly compact space $X$ and a
totally feebly compact space $Y$ is selectively feebly compact.
\end{proposition}
\begin{proof}
Let $\{U_n\}_{n\in\omega}$ be a sequence of open subsets of $X{\times}Y$. For each $n\in\omega$
pick a non-empty open subsets $U^1_n$ of $X$ and $U^2_n$ of $Y$ such that
$U^1_n{\times}U^2_n\subset U_n$. Taking a subsequence, if needed, we can assume that that there
exists a compact subset $K$ of the space $Y$ intersecting each set $U^2_n$, $n\in\omega$. Since $X$ is
selectively feebly compact, for each $n\in\w$ we can choose a point $x_n\in U^1_n$  such that a
sequence $\{x_n\}_{n\in\omega}$ has an accumulation point $x\in X$.
For each $n\in\omega$ pick a point
$y_n\in U^2_n\cap K$. Then $(x_n,y_n)\in U^1_n{\times}U^2_n\subset U_n$.
Let $\mathcal B(x)$ be the family of neighborhoods of the point $x$.
For each $U\in \mathcal B(x)$ put $Y_U=\overline{\{y_n \mid x_n\in U\}}$.
Then $\{Y_U \mid U\in \mathcal B(x)\}$ is a centered family of closed subsets of a compact space $K$, so
there exists a point $y\in\bigcap \{Y_U \mid U\in \mathcal B(x)\}$. Clearly, $(x,y)$ is an accumulation point
of the sequence $\{(x_n,y_n)\}_{n\in\omega}$.
\end{proof}




An \emph{extension} of a space $X$ is a space $Y$ containing $X$ as a dense
subspace. Hausdorff extensions of topological spaces were investigated in~\cite{Dik1,Moo,Por3,Por4,Por5}. A class $\mathcal C$ of spaces is called \emph{extension closed} provided
each extension of each space of $\mathcal C$ belongs to $\mathcal C$.
If $Y$ is a space, a class $\mathcal C$ of spaces is
$Y$-\emph{productive} provided $X{\times}Y\in \mathcal C$ for each space
$X\in\mathcal C$.
It is well-known or easy to check
that each of the following classes of spaces is extension closed:
countably pracompact, $\omega$-bounded pracompact, totally countably pracompact,
feebly compact, 
selectively feebly compact,
and feebly $\omega$-bounded.
Since $[0,\omega_1)$ endowed with
the order topology is $\omega$-bounded and
sequentially compact, each of these classes is $[0,\omega_1)$-productive by
Proposition ~\ref{prop:cp-prod}, \cite[Proposition 2.4]{Gutik-Ravsky},
\cite[Proposition 2.2]{Gutik-Ravsky}, \cite[Lemma 4.2]{dow},
Proposition~\ref{prop:sfc-prod}, and Proposition ~\ref{prop:fwb-prod},
respectively.






Next we introduce a construction which helps us to construct a pseudocompact topological semigroup which contains the bicyclic monoid as a closed subsemigroup providing a positive answer to Problem~\ref{problem}.

Let $X$ and $Y$ be topological spaces such that there exists a continuous injection $f:X\rightarrow Y$. Then by $E_Y^f(X)$ we denote the
subset $[0,\omega_1]{\times} Y\setminus\{(\omega_1,y)\mid y\in Y\setminus f(X)\}$
of a product $[0,\omega_1]{\times} Y$ endowed with a topology $\tau$ which is defined as follows.
A subset $U\subset E_Y(X)$ is open if it satisfies the following conditions:
\begin{itemize}
\item for each $\alpha<\omega_1$, if $(\alpha,y)\in U$ then there exist $\beta<\alpha$ and an open neighborhood $V_y$ of $y$ in $Y$ such that $(\beta,\alpha]{\times}V_y\subset U$;
\item if $(\omega_1,f(x))\in U$ then there exist $\alpha<\omega_1$, an open neighborhood $V_{f(x)}$ of $f(x)$ in $Y$ and an open neighborhood $W_{x}$ of $x$ in $X$, such that $f(W_x)\subset V_{f(x)}$ and $(\alpha,\omega_1){\times}V_{f(x)}\cup \{\omega_1\}{\times}f(W_x)\subset U$.
\end{itemize}

Remark that $\{\omega_1\}{\times}f(X)$ is a closed subset of $E^f_Y(X)$ homeomorphic to $X$.

\begin{proposition}\label{th0} Let $X$ be a topological space which admits a continuous injection $f$ into a space $Y$ and
$\mathcal C$ be any extension closed, $[0,\omega_1)$-productive class of spaces. If $Y\in\mathcal C$ then $E_Y^f(X)\in\mathcal C$.
\end{proposition}
\begin{proof}
Let $Y\in\mathcal C$. Since $\mathcal C$ is $[0,\omega_1)$-productive, $[0,\omega_1)\times Y\in\mathcal C$.
A space $E_Y^f(X)$ is an extension of the space $[0,\omega_1)\times Y\in\mathcal C$ providing that
$E_Y^f(X)\in\mathcal C$.
\end{proof}

If a space $X$ is a subspace of a topological space $Y$ and $id$ is the identity embedding of $X$ into $Y$, then
by $E_Y(X)$ we denote the space $E_Y^{id}(X)$. It is easy to see that $E_Y(X)$ is a subspace of a product $[0,\omega_1]{\times}Y$ which implies that if $Y$ is Tychonoff then so is $E_Y(X)$.

\begin{proposition}\label{th} Let $X$ be a subspace of a pseudocompact space $Y$. Then $E_Y(X)$ is pseudocompact and contains a closed copy of $X$.
\end{proposition}
\begin{proof}
The above arguments imply that $E_Y(X)$ is Tychonoff. Fix any continuous real valued function $f$ on $E_Y(X)$.
Observe that the dense subspace $[0,\omega_1)\times Y$ of $E_Y(X)$ is pseudocompact. Then the restriction of $f$ on the subset $[0,\omega_1)\times Y$ is bounded, i.e., there exist reals $a,b$ such that $f([0,\omega_1)\times Y)\subset [a,b]$. Then $f^{-1}[a,b]$ is closed in $E_Y(X)$ and contains the dense subset $[0,\omega_1)\times Y$ witnessing that $f^{-1}[a,b]=E_Y(X)$.
Hence the space $E_Y(X)$ is pseudocompact.
\end{proof}

Embeddings into countable compact and $\omega$-bounded topological spaces were investigated in~\cite{BBR1,BBR2}.

A family $\mathcal{A}$ of countable subsets of a set $X$ is called {\em almost disjoint} if for
each $A,B\in \mathcal{A}$ the set $A\cap B$ is finite. Given a property $P$, the almost disjoint
family $\mathcal{A}$ is called {\em $P$-maximal} if each element of $\mathcal{A}$ has the property
$P$ and for each countable subset $F\subset X$ which has the property $P$ there exists
$A\in\mathcal{A}$ such that the set $A\cap F$ is infinite.

Let $\mathcal{F}$ be a family of closed subsets of a topological space $X$. The topological space $X$ is called
\begin{itemize}
\item {\em $\mathcal{F}$-regular}, if for any set $F\in\mathcal{F}$ and point $x\in X\setminus F$ there exist disjoint open sets $U,V\subset X$ such that $F\subset U$ and $x\in V$;
\item {\em $\mathcal{F}$-normal}, if for any disjoint sets $A,B\in\mathcal{F}$ there exist disjoint open sets $U,V\subset X$ such that $A\subset U$ and $B\subset V$.
\end{itemize}

Given a topological space $X$, by $\mathcal{D}_{\omega}$ we denote the family of countable closed
discrete subsets of $X$. We say that a subset $A$ of $X$ satisfies a property
$\mathcal{D}_{\omega}$ iff $A\in \mathcal{D}_{\omega}$.

\begin{theorem}\label{thc}
Each $\mathcal{D}_{\omega}$-regular topological space $X$ can be embedded as an open dense subset into a countably pracompact topological space.
\end{theorem}

\begin{proof}
By Zorn's Lemma, there exists a $\mathcal{D}_{\omega}$-maximal almost disjoint family $\mathcal{A}$ on
$X$. Let $Y=X\cup \mathcal{A}$. We endow $Y$ with the topology $\tau$
defined as follows. A subset $U\subset Y$ belongs to $\tau$ iff it satisfies the following
conditions:
\begin{itemize}
\item if $x\in U\cap X$, then there exists an open neighborhood $V$ of $x$ in $X$ such that $V\subset U$;
\item if $A\in U\cap \mathcal{A}$, then there exists a cofinite subset $A^{'}\subset A$ and an open set $V$ in $X$ such that $A^{'}\subset V\subset U$.
\end{itemize}
Observe that $X$ is an open dense subset of $Y$ and $\mathcal{A}$ is discrete and closed in $Y$. Since $X$ is $\mathcal{D}_{\omega}$-regular for each distinct points $x\in X$ and $y\in Y$ there exist disjoint open neighborhoods $U_x$ and $U_y$ in $Y$. By Proposition 2.1 from~\cite{BBR1}, each $\mathcal{D}_{\omega}$-regular topological space is $\mathcal{D}_{\omega}$-normal. Fix any distinct $A,B\in\mathcal{A}$. Put $A^{'}=A\setminus (A\cap B)$ and $B^{'}=B\setminus (A\cap B)$. By the $\mathcal{D}_{\omega}$-normality of $X$ there exist disjoint open neighborhoods $U_{A^{'}}$ and $U_{B^{'}}$ of $A^{'}$ and $B^{'}$, respectively. Then the sets $U_A=\{A\}\cup U_{A^{'}}$ and $U_B=\{B\}\cup U_{B^{'}}$ are disjoint open neighborhoods of $A$ and $B$, respectively, in $Y$. Hence the space $Y$ is Hausdorff.

Observe that the maximality of the family $\mathcal{A}$ implies that there exists no countable discrete subset $D\subset X$ which is closed in $Y$. Hence each infinite subset $A$ in $X$ has an accumulation point in $Y$, that is, $Y$ is countably pracompact.
\end{proof}

However, there exists a Hausdorff topological space which cannot be embedded as a dense open subset into countably pracompact topological spaces.
\begin{example}
Let $\tau$ be the usual topology on the real line $\mathbb{R}$ and $C=\{A\subset \mathbb{R}: |\mathbb{R}\setminus A|\leq \omega\}$. By $\tau^*$ we denote the topology on $\mathbb{R}$ which is generated by the subbase $\tau\cup C$. Obviously, the space $\mathbb{R}^*=(\mathbb{R},\tau^*)$ is Hausdorff. We claim that $\mathbb{R}^*$ cannot be embedded as a dense open subset into a countably pracompact topological space. Assuming the contrary, let $X$ be a countably pracompact topological space which contains $\mathbb{R}^*$  as a dense open subspace. Since $X$ is countably pracompact there exists a dense subset $Y$ of $X$ such that each infinite subset of $Y$ has an accumulation point in $X$. Since $\mathbb{R}^*$ is open and dense in $X$ the set $Z=\mathbb{R}^*\cap Y$ is dense in $X$. Moreover, it is
dense in $(\mathbb{R},\tau)$. Fix any point $z\in Z$ and a sequence $\{z_n\}_{n\in\omega}$
of distinct points of $Z\setminus\{z\}$ converging to $z$ in $(\mathbb{R},\tau)$. Since $Z$ is
dense in $(\mathbb{R},\tau)$ such a sequence exists. Observe that $\{z_n\}_{n\in\omega}$ is closed
and discrete in $\mathbb{R}^*$. So its accumulation point $x$ belongs to $X\setminus \mathbb{R}^*$.
Observe that for each open neighborhood $U$ of $z$ in $\mathbb{R}^*$ all but finitely many $z_n$ belongs to the closure of $U$. Hence $x\in\overline{U}$ for each open neighborhood $U$ of $z$ which contradicts to the Hausdorffness of $X$.
\end{example}


\begin{theorem}\label{thcp}
Each topological space can be embedded as a closed subset into an $\omega$-bounded pracompact topological space.
\end{theorem}

\begin{proof}
Let $X$ be a topological space. By $X_d$ we denote the set $X$ endowed with a discrete topology. Let $X^*$ be the one point compactification of the  space $X_d$. The unique non-isolated point of $X^*$ is denoted by $\infty$. Put $Y=[0,\omega_1]{\times}X^*\setminus \{(\omega_1,\infty)\}$. We endow $Y$ with a topology $\tau$ defined as follows. A subset $U$ is open in $(Y,\tau)$ if it satisfies the following conditions:
\begin{itemize}
\item if $(\alpha, \infty)\in U$, then there exist $\beta<\alpha$ and a cofinite subset $A$ of $X^*$ which contains $\infty$ such that $(\beta,\alpha]{\times}A\subset U$;
\item if $(\omega_1,x)\in U$, then there exist $\alpha<\omega_1$ and an open (in $X$) neighborhood $V$ of $x$ such that $(\alpha,\omega_1]{\times}V\subset U$.
\end{itemize}
Observe that the subset $[0,\omega_1){\times} X^*\subset Y$ is open, dense and $\omega$-bounded. Hence $Y$ is $\omega$-bounded pracompact. It is easy to see that the subset $\{\omega_1\}{\times}X\subset Y$ is closed and homeomorphic to $X$.
\end{proof}

Next we introduce a construction which helps us to prove that any space can be embedded as a closed subspace into an H-closed topological space.

Denote the subspace $\{1-1/n\mid n\in\N\}\cup\{1\}$ of the real line by
$J$.
Let $X$ be a dense open subset of a topological space $Y$. By $Z$ we denote the set
$(J{\times}Y)\setminus\{(t,y)\mid y\in Y\setminus X\hbox{ and } t>0\}.$
By $H_Y(X)$ we denote the set $Z$ endowed with a topology defined as follows.
A subset $U\subset Z$ is open in $H_Y(X)$ if it satisfies the following conditions:
\begin{itemize}
\item for each $x\in X$ if $(t,x)\in U$, then there exist open neighborhoods $V_t$ of $t$ in $J$ and $V_x$ of $x$ in $X$ such that $V_t{\times}V_x\subset U$;
\item for each $y\in Y\setminus X$ if $(0,y)\in U$, then there exists an open neighborhood $V_y$ of $y$ in $Y$ such that $\{0\}{\times} (V_y\setminus X)\cup (J\setminus\{1\}){\times}(V_y\cap X)\subset U$.
\end{itemize}

Obviously, the space $H_Y(X)$ is Hausdorff and the subset $\{(1,x)\mid x\in X\}\subset H_Y(X)$ is closed and homeomorphic to $X$.

\begin{proposition}\label{H-cl}
If $Y$ is an H-closed topological space, then $H_Y(X)$ is H-closed.
\end{proposition}

\begin{proof}
Fix an arbitrary filter $\mathcal{F}$ on $H_Y(X)$.
One of the following three cases holds:
\begin{itemize}
\item[(1)] there exists $t\in J\setminus\{1\}$ such that for each $F\in \mathcal{F}$ there exists $y\in Y$ such that $(t,y)\in F$;
\item[(2)] for each $F\in \mathcal{F}$ there exists $x\in X$ such that $(1,x)\in F$;
\item[(3)] for every $t\in J$ there exists $F\in \mathcal{F}$ such that $(t,y)\notin F$ for each $y\in Y$.
\end{itemize}

Consider case (1).
For each $F\in \mathcal{F}$ put $F_t=F\cap (\{t\}{\times}X\cup \{0\}{\times}(Y\setminus X))$.
Clearly, a family $\mathcal{F}_{t}=\{F_t\mid F\in \mathcal{F}\}$ is a filter on $\{t\}{\times}X\cup \{0\}{\times}(Y\setminus X)$.
Observe that for each $t\in J\setminus\{1\}$ the subspace $\{t\}{\times}X\cup \{0\}{\times}(Y\setminus X)$ is homeomorphic to $Y$ and hence is H-closed.
Then there exists a $\theta$-accumulation point $z\in \{t\}{\times}X\cup \{0\}{\times}(Y\setminus X)$ of the filter $\mathcal{F}_t$.
Obviously, $z$ is
a $\theta$-accumulation point of the filter $\mathcal{F}$.

Consider case (2). For each $F\in \mathcal{F}$ put $F_0=\{(0,x)\mid (1,x)\in F\}$. Clearly, the family $\mathcal{F}_{0}=\{F_0\mid F\in \mathcal{F}\}$ is a filter on the H-closed space $\{0\}{\times} Y$. Hence there exists $y\in Y$ such that $(0,y)$ is a $\theta$-accumulation point of the filter $\mathcal{F}_{0}$. If $y\in X$, then $(1,y)$ is a $\theta$-accumulation point of the filter $\mathcal{F}$. If $y\in Y\setminus X$, then we claim that $(0,y)$ is a $\theta$-accumulation point of the filter $\mathcal{F}$.
Indeed, let $U$ be any open neighborhood of the point $(y,0)$.
There exists an open neighborhood $V_y$ of $y$ in $Y$ such that
$V=\{0\}{\times} (V_y\setminus X)\cup (J\setminus\{1\}){\times}(V_y\cap X)\subset U$.
Since $(0,y)$ is a $\theta$-accumulation point of the filter $\mathcal{F}_{0}$,
$\overline{V}\cap F_0\neq \emptyset$ for each $F_0\in \mathcal{F}_0$.
Fix any $F\in \mathcal{F}$ and $(0,z)\in \overline{V}\cap F_0$. The definition of the topology on
$H_Y(X)$ yields that the set $\{(t,z)\mid t\in J\setminus\{1\}\}$ is contained in $\overline{V}$.
Then $(1,z)\in\overline{\{(t,z)\mid t\in J\setminus\{1\}\}}\subset\overline{V}$.
Hence for each $F\in \mathcal{F}$ the set
$\overline{U}\cap F$ is non-empty providing that $(0,y)$ is a
$\theta$-accumulation point of the filter $\mathcal{F}$.

Consider case (3).
For each $F\in \mathcal{F}$ denote
$$F^*=\{(0,x)\mid \hbox{there exists }t\in I \hbox{ such that }(t,x)\in F\}.$$
Let $(0,y)$ be a $\theta$-accumulation point of the
filter $\mathcal{F}^{*}=\{F^*\mid F\in \mathcal{F}\}$.

If $y\in X$, then we claim that $(1,y)$ is a
$\theta$-accumulation point of the filter $\mathcal{F}$. Indeed fix any $F\in \mathcal{F}$ and a basic open neighborhood
$V=\{t\in J\mid t> 1-1/n\}{\times} U$ of $(1,y)$ where $n$ is some fixed positive integer and $U$ is an open neighborhood of $y$ in $X$. By the assumption, there exist sets $F_0,\ldots,F_n\subset \mathcal{F}$ such that $F_i\cap \{(1/i,x)\mid x\in Y\}=\emptyset$ for every $i\leq n$. Then the set $H=\cap_{i\leq n}F_i\cap F$ belongs to $\mathcal{F}$ and for each $(t,x)\in H$,  $t>1-1/n$.
Since $(0,y)$ is a $\theta$-accumulation point of the
filter $\mathcal{F}^{*}$ the set $\overline{\{0\}\times U}\cap H^*$
is non-empty. Fix any $(0,x)\in\overline{\{0\}\times U}\cap H^*$.
Then there exists $k>n$ such that $(1-1/k,x)\in H\subset F$. The definition of the topology on the space $H_Y(X)$ implies that $(1-1/k,x)\in \overline{V}\cap H\subset \overline{V}\cap F$ which implies that $(1,y)$ is a $\theta$-accumulation point of the filter $\mathcal{F}$.

If $y\in Y\setminus X$, then even more simple arguments show
that $(0,y)$ is a $\theta$-accumulation point of the filter $\mathcal{F}$.

Hence the space $H_Y(X)$ is H-closed.
\end{proof}

\begin{theorem}\label{th2}
For any topological space $X$ there exists an H-closed space $Z$ which contains $X$ as a closed subspace.
\end{theorem}

\begin{proof}
For each Hausdorff topological space $X$ there exists an H-closed space $Y$ which contains $X$ as a dense open subspace
(see~\cite[Problem 3.12.6]{Engelking-1989}). By Proposition~\ref{H-cl}, the space $H_Y(X)$ is H-closed.
It remains to note that the set $\{(1,x)\mid x\in X\}\subset H_Y(X)$ is closed and homeomorphic to $X$.
\end{proof}

\section{Applications for topological semigroups}

A semigroup $S$ is called an \emph{inverse semigroup}, if for each element $a\in S$ there exists a unique
element $a^{-1}\in S$ such that $aa^{-1}a=a$ and $a^{-1}aa^{-1}=a^{-1}$.
The map which associates every element of an inverse semigroup to its
inverse is called an \emph{inversion}.

A topological (inverse) semigroup is a topological space endowed with a
continuous semigroup operation (and an~inversion, resp.).
In this case the topology of the space is called (\emph{inverse}, resp.)
\emph{semigroup topology}.
A semitopological semigroup is a topological space endowed
with a separately continuous semigroup operation. It this case the topology
of the space is called \emph{shift-continuous}.

Let $X$ be a non-empty set. By $\mathcal{B}_{X}$ we denote the set
 $
 X{\times}X\cup\{0\}
 $
where $0\notin X{\times}X$ endowed with the following semigroup
operation:
\begin{equation*}
\begin{split}
&(a,b)\cdot(c,d)=
\left\{
  \begin{array}{cl}
    (a,d), & \hbox{ if~ } b=c;\\
    0, & \hbox{ if~ } b\neq c,
  \end{array}
\right.\\
&\hbox{and } (a,b)\cdot 0=0\cdot(a,b)=0\cdot 0=0, \hbox{ for each } a,b,c,d\in X.
\end{split}
\end{equation*}
The semigroup $\mathcal{B}_{X}$ is called the \emph{semigroup of $X{\times}X$-matrix units}. Observe that semigroups $\mathcal{B}_{X}$ and $\mathcal{B}_{Y}$ are isomorphic iff $|X|=|Y|$.

If a set $X$ is infinite then the semigroup of $X{\times}X$-matrix units cannot be embedded into a compact topological semigroup 
(see \cite[Theorem 3]{Gutik-2005}). In ~\cite[Theorem~5]{Gutik-2009} this result was generalized for countably compact topological semigroups. 
Moreover, in~\cite[Theorem~4.4]{BardGut-2016(1)} it was shown that for an infinite set $X$ the semigroup $\mathcal{B}_{X}$ cannot be embedded densely into a feebly compact topological semigroup.

A bicyclic monoid $\mathcal{C}(p,q)$ is the semigroup with the identity $1$ generated by two elements $p$ and $q$ subject to the condition $pq=1$.
The bicyclic monoid is isomorphic to the set $\omega{\times}\omega$ endowed with the following semigroup operation:

\begin{equation*}
(a,b)\cdot(c,d)=
\left\{
  \begin{array}{cl}
    (a+c-b,d), & \hbox{ if~ } b\leq c;\\
    (a,d+b-c), & \hbox{ if~ } b> c.
  \end{array}
\right.
\end{equation*}

Neither stable nor $\Gamma$-compact topological semigroups can
contain a copy of the bicyclic monoid (see \cite{Anderson-Hunter-Koch-1965,Hildebrant-Koch-1988}).
In~\cite{GutRep-2007} it was proved that the bicyclic monoid does not embed into a countably
compact topological inverse semigroup. Also a topological semigroup with a pseudocompact square
cannot contain the bicyclic monoid~\cite{BanDimGut-2010}.
On the other hand, in~\cite[Theorem 6.1]{BanDimGut-2010} it was proved that there exists a
Tychonoff countably pracompact topological semigroup $S$ densely containing the bicyclic monoid.
Moreover, under Martin's Axiom the semigroup $S$ is countably compact
(see~\cite[Theorem 6.6 and Corollary 6.7]{BanDimGut-2010}).
However, it is still unknown whether there exists under ZFC a countably compact topological semigroup
containing the bicyclic monoid (see~\cite[Problem 7.1]{BanDimGut-2010}).
Also, in~\cite{BanDimGut-2010} the following problem was posed:

\begin{problem}[{\cite[Problem 7.2]{BanDimGut-2010}}]\label{problem}
Is there a pseudocompact topological semigroup $S$ that contains a closed copy of the bicyclic monoid?
\end{problem}


Embeddings of semigroups which are generalizations of the bicyclic monoid into compact-like topological semigroups were investigated in~\cite{Bardyla-2019(1),BardGut-2016(1)}.
Namely, in~\cite{BardGut-2016(1)} it was proved that for each cardinal $\lambda>1$ a polycyclic monoid
$\mathcal{P}_{\lambda}$ does not embed as a dense subsemigroup into a feebly compact topological
semigroup. In~\cite{Bardyla-2019(1)} were described graph inverse semigroups which embed densely
into feebly compact topological semigroups.

Observe that the space $[0,\omega_1]$ endowed with a semigroup operation of taking minimum becomes a
topological semilattice and therefore a topological inverse semigroup.
\begin{lemma}\label{l1}
Let $X$ and $Y$ be semitopological (topological, topological inverse, resp.) semigroups  such that there exists a continuous injective homomorphism $f:X\rightarrow Y$. Then $E_Y^f(X)$ is a semitopological (topological, topological inverse, resp.) semigroup with respect to the semigroup operation inherited from a direct product of semigroups $(\omega_1,\min)$ and $Y$.
\end{lemma}
\begin{proof}
We prove this lemma for the case of topological semigroups $X$ and $Y$. Proofs in other cases are similar. Fix any elements $(\alpha,x), (\beta,y)$ of $E_Y^f(X)$.  Also, assume that
$\beta\leq \alpha$. In the other case the proof will be similar.
Fix any open neighborhood $U$ of $(\beta,xy)=(\alpha,x)\cdot(\beta,y)$. There are three cases to consider:
\begin{itemize}
\item[(1)] $\beta\leq \alpha<\omega_1$;
\item[(2)] $\beta< \alpha=\omega_1$;
\item[(3)] $\alpha=\beta=\omega_1$.
\end{itemize}
In case (1) there exist $\gamma<\beta$ and an open neighborhood $V_{xy}$ of $xy$ in $Y$ such that $(\gamma,\beta]{\times}V_{xy}\subset U$.  Since $Y$ is a topological semigroup there exist open neighborhoods $V_x$ and $V_y$ of $x$ and $y$, respectively, such that $V_x\cdot V_y\subset V_{xy}$. Put $U_{(\alpha,x)}=(\gamma,\alpha]{\times}V_x$ and $U_{(\beta,y)}=(\gamma,\beta]{\times}V_y$. It is easy to check that $U_{(\alpha,x)}\cdot U_{(\beta,y)}\subset (\gamma,\beta]{\times}V_{xy}\subset U$.

Consider case (2). Similarly as in case (1) there exist an ordinal $\gamma<\beta$ and open neighborhoods $V_x$, $V_y$ and $V_{xy}$ of $x,y$ and $xy$, respectively, such that $(\gamma,\beta]{\times}V_{xy}\subset U$ and $V_x\cdot V_y\subset V_{xy}$.
Since the map $f$ is continuous there exists an open neighborhood $V_{f^{-1}(x)}$ of $f^{-1}(x)$ in $X$ such that $f(V_{f^{-1}(x)})\subset V_x$. Put $U_{(\omega_1,x)}=(\beta,\omega_1){\times}V_x\cup \{\omega_1\}{\times}f(V_{f^{-1}(x)})$ and $U_{(\beta,y)}=(\gamma,\beta]{\times}V_y$. It is easy to check that $U_{(\omega_1,x)}\cdot U_{(\beta,y)}\subset (\gamma,\beta]{\times}V_{xy}\subset U$.

Consider case (3). There exist ordinal $\gamma<\omega_1$, an open neighborhood $V_{xy}$ of $xy$ in $Y$ and an open neighborhood $W_{f^{-1}(xy)}$ of $f^{-1}(xy)$ in $X$ such that
$(\gamma,\omega_1){\times}V_{xy}\cup \{\omega_1\}{\times}f(W_{f^{-1}(xy)})\subset U.$

Since $Y$ is a topological semigroup there exist open (in $Y$) neighborhoods $V_x$ and $V_y$ of $x$ and $y$, respectively, such that $V_x\cdot V_y\subset V_{xy}$.
Since the map $f$ is continuous and $X$ is a topological semigroup there exist open (in $X$) neighborhoods $W_{f^{-1}(x)}$ and $W_{f^{-1}(y)}$ of $f^{-1}(x)$ and $f^{-1}(y)$, respectively, such that $W_{f^{-1}(x)}\cdot W_{f^{-1}(y)}\subset W_{f^{-1}(xy)}$, $f(W_{f^{-1}(x)})\subset V_x$ and $f(W_{f^{-1}(y)})\subset V_y$.
Put $U_{(\omega_1,x)}=(\gamma,\omega_1){\times}V_x\cup \{\omega_1\}{\times}f(W_{f^{-1}(x)})$ and $U_{(\omega_1,y)}=(\gamma,\omega_1){\times}V_y\cup \{\omega_1\}{\times}f(W_{f^{-1}(y)})$.
It is easy to check that $U_{(\omega_1,x)}\cdot U_{(\omega_1,y)}\subset U$.

Hence the semigroup operation in $E_Y^f(X,\tau_X)$ is continuous.
\end{proof}

\begin{remark}\label{r}
The subsemigroup $\{(\omega_1,f(x))\mid x\in X\}\subset E_Y^f(X)$ is closed and topologically isomorphic to $X$.
\end{remark}


Proposition~\ref{th0}, Lemma~\ref{l1} and Remark~\ref{r} imply the following:

\begin{proposition}\label{cc1}
Let $X$ be a (semi)topological semigroup which admits a continuous injective homomorphism $f$ into a (semi)topological semigroup $Y$ and
$\mathcal C$ be any $[0,\omega_1)$-productive, extension closed class of spaces.
If $Y\in\mathcal C$ then the (semi)topological semigroup $E^f_Y(X)\in \mathcal C$
and contains a closed copy of a (semi)topological semigroup $X$.
\end{proposition}

Proposition~\ref{th}, Lemma~\ref{l1} and Remark~\ref{r} imply the following:

\begin{proposition}\label{cc2}
Let $X$ be a subsemigroup of a pseudocompact (semi)topological semigroup $Y$.
Then the (semi)topological semigroup $E_Y(X)$ is pseudocompact
and contains a closed copy of the (semi)\-to\-po\-lo\-gi\-cal semigroup $X$.
\end{proposition}

By~\cite[Theorem 6.1]{BanDimGut-2010}, there exists a Tychonoff countably pracompact
(and hence pseudocompact) topological semigroup $S$ containing densely the bicyclic monoid. Hence Proposition~\ref{cc2} implies the following corollary which gives a positive answer to Problem~\ref{problem}.
\begin{corollary}
There exists a pseudocompact topological semigroup which contains a closed copy of the bicyclic monoid.
\end{corollary}

Further we will need the following definitions.
A subset $A$ of a topological space is called $\theta$-closed if for each element $x\in X\setminus A$ there exists an open neighborhood $U$ of $x$ such that $\overline{U}\cap A=\emptyset$. Observe that if a topological space $X$ is regular then each closed subset $A$ of $X$ is $\theta$-closed.  A topological space $X$ is called {\em weakly H-closed} if each $\w$-filter $\mathcal{F}$ has a $\theta$-accumulation point in $X$. Weakly H-closed spaces were investigated in~\cite{Os1}. Obviously, for a topological space $X$ the following implications hold: $X$ is H-closed $\Rightarrow$ $X$ is weakly H-closed $\Rightarrow$ $X$ is feebly compact. However, neither of the above implications can be inverted. Indeed, an arbitrary pseudocompact but not countably compact space will be an example of feebly compact space which is not weakly H-closed. The space $[0,\w_1)$ with an order topology is an example of weakly H-closed but not H-closed space.

The following theorem shows that Theorem~\ref{th2} cannot be generalized for topological semigroups.
\begin{theorem}\label{matr}
The semigroup  $\mathcal{B}_{\w}$ of $\w{\times}\w$-matrix units does not embed into a weakly H-closed
topological semigroup.
\end{theorem}
\begin{proof}
Suppose to the contrary that $\mathcal{B}_{\w}$ is a subsemigroup of a weakly H-closed topological semigroup $S$. By $E(\mathcal{B}_{\w})$ we denote the semilattice of idempotents of $\mathcal{B}_{\w}$. Observe that $E(\mathcal{B}_{\w})=\{(n,n)\mid n\in\w\}$ and all maximal chains of  $E(\mathcal{B}_{\w})$ contain two elements. Then, by~\cite[Theorem~2.1]{BBm}, the set $E(\mathcal{B}_{\w})$ is $\theta$-closed in $S$.
Let $\mathcal{F}$ be an arbitrary $\w$-filter on the set $\{(n,n)\mid n\in\w\}$.
Since $S$ is weakly H-closed then there exists a $\theta$-accumulation point $s\in S$ of the filter $\mathcal{F}$. Since the set $E(\mathcal{B}_{\w})$ is $\theta$-closed in $S$ we obtain that $s\in E(\mathcal{B}_{\w})$. We show that $s=0$. It is sufficient to prove that $s\cdot s=0$. Fix an arbitrary open neighborhood $W$ of $s\cdot s$.
Since $S$ is a topological semigroup
there exists an open neighborhood $V_s$ of $s$ such that
$\overline{V_s}\cdot \overline{V_s}\subset \overline{W}$. Since $s$ is a $\theta$-accumulation point of the filter $\mathcal{F}$ there exists infinitely many $n\in \omega$ such that $(n,n)\in \overline{V_s}$. Fix two distinct elements $(n_1,n_1)\in \overline{V_s}$ and $(n_2,n_2)\in \overline{V_s}$. The definition of the semigroup operation in $\mathcal{B}_{\w}$ implies that $0=(n_1,n_1)\cdot (n_2,n_2)\in \overline{V_s}\cdot \overline{V_s}\subset \overline{W}$. Hence $0\in \overline{W}$ for each open neighborhood $W$ of $s\cdot s$ which implies that $s\cdot s=0$.
Hence for each $\w$-filter $\mathcal{F}$ on the set $\{(n,n)\mid n\in\w\}$ the only $\theta$-accumulation point of $\mathcal{F}$ is
$0$. Thus for each open neighborhood $U$ of $0$ the set $A_U=\{(n\mid (n,n)\notin \overline{U}\}$ is
finite, because if there exists an open neighborhood $U$ of $0$ such that the set $A_U$ is infinite, then $0$ is not a $\theta$-accumulation point of the $\w$-filter $\mathcal{F}$ which has a base consisting of cofinite subsets of $A_U$.

Let $\mathcal{F}$ be an arbitrary $\w$-filter on the set $\{(1,n)\mid n\in\w\}$.
Since $S$ is weakly H-closed there exists a $\theta$-accumulation point $s\in S$ of the
filter $\mathcal{F}$.

We claim that $s\cdot 0=0$. Indeed, fix any open neighborhood $W$ of $s\cdot 0$. The continuity of the semigroup operation in $S$ yields open neighborhoods $V_s$ of $s$ and $V_0$ of $0$ such that $\overline{V_s}\cdot \overline{V_0}\subset \overline{W}$. Since the set $A_{V_0}=\{(n\mid (n,n)\notin \overline{V_0}\}$ is
finite and $s$ is a $\theta$-accumulation point of the filter $\mathcal{F}$ there exist distinct $n,m\in\omega$ such that $(1,n)\in \overline{V_s}$ and $(m,m)\in \overline{V_0}$. Then $0=(1,n)\cdot (m,m)\in \overline{V_s}\cdot \overline{V_0}\subset \overline{W}$. Hence $0\in\overline{W}$ for each open neighborhood $W$ of $s\cdot 0$ witnessing that $s\cdot 0=0$.

Fix an arbitrary open neighborhood $U$ of $0$. Since $s\cdot 0=0$
and $S$ is a topological semigroup, there exist
open neighborhoods $V_s$ of $s$ and $V_0$ of $0$ such that $\overline{V_s}\cdot \overline{V_0}\subset \overline{U}$. Recall that the set $\{n\mid (n,n)\notin \overline{V_0}\}$ is finite. Then $(1,n)=(1,n)\cdot (n,n)\in \overline{V_s}\cdot \overline{V_0}\subset \overline{U}$ for all but finitely many elements $(1,n)\in \overline{V_s}$. Hence $0$ is a $\theta$-accumulation point of the $\w$-filter $\mathcal{F}$. Since the filter $\mathcal{F}$ was selected arbitrarily, $0$ is a $\theta$-accumulation point of any $\w$-filter on the set $\{(1,n)\mid n\in\w\}$. As a consequence, for each open neighborhood $U$ of $0$ the set $B_U=\{n\mid (1,n)\notin \overline{U}\}$ is finite.

Similarly it can be shown that for each open neighborhood $U$ of $0$ the set $C_U=\{n\mid (n,1)\notin \overline{U}\}$ is finite.

Fix an open neighborhood $U$ of $0$ such that $(1,1)\notin \overline{U}$.
Since $0=0\cdot 0$ the continuity of the semigroup operation implies that there exists an open neighborhood $V$ of $0$ such that $\overline{V}\cdot \overline{V}\subset \overline{U}$. The finiteness of the sets $B_V$ and $C_V$ implies that there exists $n\in \w$ such that $\{(1,n),(n,1)\}\subset \overline{V}$. Hence $(1,1)=(1,n)\cdot (n,1)\in \overline{V}\cdot \overline{V}\subset \overline{U}$,
which contradicts to the choice of $U$.
\end{proof}

\begin{corollary}\label{matr1}
The semigroup of $\w{\times}\w$-matrix units does not embed into a topological semigroup $S$ which is an H-closed topological space.
\end{corollary}

However, we have the following questions:

\begin{question}
Does there exist a feebly compact topological semigroup $S$ which contains a semigroup of $\w{\times}\w$-matrix units?
\end{question}

\begin{question}~\label{q}
Does there exist a topological semigroup $S$ which cannot be embedded into a feebly compact topological semigroup $T$?
\end{question}


We remark that these questions were posed at the Lviv Topological Algebra Seminar a few years ago. 





\begin{thebibliography}{00}




\bibitem{Anderson-Hunter-Koch-1965}
L. Anderson, R. Hunter, and R. Koch,
{\em Some results on stability in semigroups.}
Trans. Amer. Math. Soc. {\bf 117} (1965), 521--529.






\bibitem{BBm} T.~Banakh, S.~Bardyla,
\emph{Characterizing chain-finite and chain-compact topological semilattices}, Semigroup Forum, {\bf 98}(2), (2019), 234--250.

\bibitem{BBR1} T. Banakh, S. Bardyla, A. Ravsky, {\em Embedding topological spaces into Hausdorff $\w$-bounded spaces}, preprint, arXiv:1906.00185.

\bibitem{BBR2} T. Banakh, S. Bardyla, A. Ravsky, {\em Embeddings into countably compact Hausdorff spaces}, preprint, arXiv:1906.04541.


\bibitem{BanDimGut-2010}
T. Banakh, S. Dimitrova, O. Gutik,
\emph{Embedding the bicyclic semigroup into countably compact topological semigroups},
Topology Appl. \textbf{157}(18) (2010), 2803--2814.






\bibitem{Bardyla-2019(1)}
S.~Bardyla,
\emph{Embeddings of graph inverse semigroups into compact-like topological semigroups},
preprint, arXiv:1810.09169.






\bibitem{BardGut-2016(1)}
S. Bardyla, O. Gutik,
\emph{On a semitopological polycyclic monoid},
Algebra Discr. Math. \textbf{21} (2016), no. 2, 163--183.









\bibitem{Dik}
D. Dikranjan, E. Giuli,
{\em $S(n)$-$\theta$-closed spaces},
Topology Appl. {\bf 28} (1988), 59--74.

\bibitem{Dik1}
D. Dikranjan, J. Pelant, {\em Categories of topological spaces with sufficiently many sequentially closed spaces,}
Cahiers de Topologie et Geometrie Differentielle Categoriques, {\bf 38}:4, (1997), 277--300.

\bibitem{dow}
A. Dow, J.R. Porter, R.M. Stephenson, Jr., and R.G. Woods,
{\em Spaces whose pseudocompact subspaces are closed subsets},
Appl. Gen. Topol. {\bf 5}:2 (2004), 243--264.





\bibitem{Engelking-1989}
R.~Engelking,
\emph{General Topology}, 2nd ed., Heldermann, Berlin, 1989.



\bibitem{Gutik-2005}
O.~Gutik and K.~Pavlyk, \emph{Topological semigroups of matrix units.}
Algebra Discr. Math. \textbf{3} (2005), 1--17.

\bibitem{Gutik-2009}
O.~Gutik, K.~Pavlyk, and A.~Reiter, \emph{Topological semigroups of matrix units and countably compact Brandt $\lambda^0$-extensions.}
 Mat. Stud. \textbf{32}:2 (2009), 115--131.







\bibitem{Gutik-Ravsky}
O.~Gutik, A.~Ravsky,
\emph{On old and new classes of feebly compact spaces},
Visn. Lviv. Univ. Ser. Mech. Math. \textbf{85}, (2018), 48--59.




\bibitem{GutRep-2007}
O.~Gutik, D.~Repovs
\emph{ On countably compact 0-simple topological inverse semigroups},
Semigroup Forum. \textbf{75}:2 (2007), 464--469.






\bibitem{Hildebrant-Koch-1988}
J. Hildebrant, R. Koch,
\emph{Swelling actions of $\Gamma$-compact semigroups},
Se\-mi\-group Forum {\bf 33}:1 (1986), 65--85.

\bibitem{Jos1}
J. E. Joseph,
\emph{On H-closed spaces},
Proceedings of the American Mathematical Society {\bf 55}:1 (1976), 223--226.


\bibitem{Jos2}
J. E. Joseph,
\emph{More characterizations of H-closed spaces},
Proceedings of the American Mathematical Society {\bf 63}:1 (1977), 160--164.

\bibitem{Moo}
D. D. Mooney,
\emph{Spaces with unique Hausdorff extensions},
Topology Appl. {\bf 61}:1 (1995), 241--256.

\bibitem{Os}
A. Osipov,
\emph{Nearly H-closed spaces},
Journal of Mathematical Sciences {\bf 155}:4 (2008), 626--633.

\bibitem{Os1}
A. Osipov,
\emph{Weakly H-closed spaces},
Proceedings of the Steklov Institute of Mathematics {\bf 10}:1 (2004), S15--S17.

\bibitem{Por1}
J. Porter,
\emph{On locally H-closed spaces},
Proc. Londom Math. Soc. {\bf 20}:3 (1970), 193--204.

\bibitem{Por2}
J. Porter, J. Thomas,
\emph{On H-closed and minimal Hausdorff spaces},
Trans. Amer. Math. Soc. {\bf 138} (1969), 159--170.

\bibitem{Por3}
J. Porter, C. Votaw,
\emph{H-closed extensions I},
Gen. Topology Appl. {\bf 3} (1973), 211--224.

\bibitem{Por4}
J. Porter, C. Votaw,
\emph{H-closed extensions II},
Trans. Amer. Math. Soc. {\bf 202} (1975), 193--208.

\bibitem{Por5}
J. Porter, R. Woods,
\emph{Extensions and Absolues of Hausdorff Spaces},
Springer, Berlin, 1988, 856 pp.



\bibitem{StephensonJr1984}
R.~M.~Stephenson, Jr,
\emph{Initially $\kappa$-compact and related compact spaces},
in K. Kunen, J.E. Vaughan (eds.), Handbook of Set-Theoretic Topology, Elsevier, 1984,
pp. 603--632.



\bibitem{VaughanHSTT}
J.E.~Vaughan,
\emph{Countably compact and sequentially compact spaces},
in K. Kunen, J.E. Vaughan (eds.), Handbook of Set-Theoretic Topology, Elsevier, 1984, 569--602.

\bibitem{Vel}
N. Velichko, \emph{H-Closed Topological Spaces.}
American Mathematical Society, {\bf 78}:2, (1968), 103--118.

\bibitem{Ver}
J. Vermeer, \emph{Closed subspaces of H-closed spaces.}
Pacific Journal of Mathematics, {\bf 118}:1, (1985), 229--247.


\end{thebibliography}
\end{document}